\documentclass[12pt,twoside,english]{article}
\usepackage{amsfonts,babel,amssymb,amsmath,amsthm}
\usepackage{graphicx}

\newtheorem{proposition}{Proposition}[section]

\newtheorem{theorem}[proposition]{Theorem}

\newtheorem{remark}{Remark}[section]

\newcommand{\punto}{\,\,\cdot\,\,}
\newcommand{\ds}{\displaystyle}
\newcommand{\smallfrac}[2]{{\textstyle\frac{#1}{#2}}}
\newcommand{\jump}[1]{[\![#1]\!]}

\title{Energy estimates for Galerkin semidiscretizations of time domain boundary integral equations}

\author{Francisco--Javier Sayas\footnote{Department of Mathematical Sciences, University of Delaware,
Newark DE 19716, USA -- e--mail: {\tt fjsayas@math.udel.edu}}}

\date{\today}

\begin{document}

\maketitle

\begin{abstract}
In this paper we present a battery of results related to how
Galerkin semidiscretization in space affects some formulations of
wave scattering and propagation problems when retarded boundary
integral equations are used.
\end{abstract}

\section{Introduction}

In this paper we present several results concerning the effect of
Galerkin semidiscretization in space when applied to formulations
where time domain boundary integral equations are used. The focus
will be set on energy conservation properties (or lack thereof) and
the model equation will be the acoustic wave equation in two and
three dimensions.

While frequency domain boundary integral equations are very often
used for scattering problems and for construction of absorbing
boundary conditions, their time domain counterparts have enjoyed a
more limited success. The work of Alain Bamberger and Toung Ha-Duong
\cite{BamHaD86a}--\cite{BamHaD86b} in the mid-80s sparked
theoretical and practical activity concerning time domain integral
equations for acoustics, elastodynamics and electromagnetism. A
detailed account of this early work can be found in the review
articles \cite{HaD03} and \cite{Cos03}. It is nowadays well
understood that some of the apparently unsurmountable difficulties
in the development stable discretizations of the time domain
integral equations are related to the correct approximation of
integrals and that these equations can be advantageously used in
situations where frequency domain formulations are not competitive
(scattering of frequency rich noise) or applicable (coupling with
nonlinear problems). In this paper we will study semidiscretization
in space with general Galerkin schemes, which can be complemented
either with Galerkin in time discretization (following the
philosophy of \cite{BamHaD86a} and its consequences) or with
Convolution Quadrature in time (following the seminal work of
Christian Lubich \cite{Lub94} and its development in \cite{Ban10},
\cite{BanSau09}, \cite{Sch01}, \cite{BanLubMelTA}).

From the point of view of the type of problems we are going to
study, we can classify them in two groups:
\begin{itemize}
\item[(a)] Scattering of acoustic waves by impenetrable obstacles.
This leads to an exterior wave equation with Dirichlet or Neumann
boundary conditions on the surface of the obstacles. Instead of
dealing with the traditional (and more practical) problem of an
incident plane or spherical wave (that has no finite energy), we
will study the effect of the scattering of a free wave that
propagates from some compactly supported initial conditions.
\item[(b)] Propagation of waves originated in a locally
non-homogeneous medium surrounded by a homogeneous medium. In this
case, we will also follow the propagation of compactly supported
initial data when we surround the computational domain by an
artificial boundary where we impose non-local boundary conditions
using time domain integral equations that give an exact expression
of the corresponding Steklov-Poincar\'{e} operators.
\end{itemize}
In order to present the techniques and results in a clear way we
will first focus on the problem of scattering of waves by a
sound-soft obstacle, formulated using a single-layer potential
representation (i.e., with an indirect boundary integral method). We
will show that any Galerkin semidiscretization in space of the
corresponding retarded integral equation is stable from the point of
view that a certain energy functional is preserved over time. In
fact, we will show that the total energy that is conserved has to be
computed counting potential and kinetic energy in the exterior and
interior of the obstacle, thus giving theoretical evidence of the
fact that energy is leaked to the interior of the scatterer, but the
total balance of energy is still correct. In a second step, we
introduce the second class of problems where a locally
non-homogeneous medium is surrounded with an artificial boundary
(that also surrounds the support of the initial data) and we follow
the propagation of these initial data when we discretize in space
with a Galerkin scheme both for the variational formulation of the
wave equation in the computational domain and for a system of two
retarded boundary integral equations that are coupled with the
interior wave equation to construct an absorbing boundary condition
on the artificial interface. At the discrete level, this can be
understood as a time domain coupling of Boundary and Finite Elements
(BEM-FEM). The result that we show is very similar to the one for
the scattering problem: energy is conserved as a function of time
(independently of Galerkin discretization) as long as we count the
energy of a ghost wave in the computational domain --the energy of
this interior wave has to be computed using the material properties
of the surrounding medium and is added to the natural energy of the
total wave in the interior and exterior domains.

Once these two model situations have been analyzed, we present two
very simple extensions where the same techniques apply to establish
similar results. The first extension deals with the scattering of a
sound-hard obstacle formulated with a double-layer acoustic
potential. The second one is the construction of a tighter integral
absorbing boundary condition using two integral equations but only
one boundary unknown. We next show that there are formulations where
energy is not conserved after space discretization. Two of them are
explored: the use of a direct integral formulation (based on
Kirchhoff's formula) for the sound-soft scattering problem and the
construction of absorbing boundary conditions with only one integral
equation. While the negative results do not imply that these
discretizations are unstable, they might indicate some undesirable
effects. Actually, the direct integral formulation for the
sound-soft scattering problem, fully discretized with a Galerkin
method in space and Convolution Quadrature (CQ) in time, can be
analyzed \cite{BanLalSayIP} (See below for more on this.) As for the
one-equation absorbing boundary condition, at present time there
seems to be insufficient theoretical and practical evidence to
suggest convergence or lack thereof. {\em The extension of all the
results in this paper to linear elastic waves is straightforward.}
Extension to other types of elastic waves (viscoelastic or
poroelastic, for instance) and to electromagnetic waves will
possibly require more effort.

Apart from the inherent interest of having theoretical evidence of
energy preservation after space discretization of some integral and
integro-differential model problems for which this energy
conservation property does not seem to be evident, the type of
results that we prove in this paper has some other applications. On
the one hand, the transformation of the semidiscrete equations to a
transmission problem is allowing us to provide an analysis of the
CQ-BEM for some scattering problems \cite{BanLalSayIP}. Up to the
present time, analysis of discrete methods for retarded integral
equations was based on estimates in the Laplace domain. The
dynamical system approach that we develop in this paper can be used
to obtain better estimates of the effect of space discretization and
approximation of data in some retarded integral equations. It is
also our belief that these techniques will be advantageous for some
kind of analysis of fully discrete BEM-FEM schemes using CQ. In this
sense, this paper can be taken as the time domain counterpart of
what was done in \cite{LalSay09} in the Laplace domain and we hope
to be able to build on these two articles for the complete analysis
of the fully discrete method.

This paper offers a novel approach to the analysis of time-domain
integral equations and their semidiscretization in space.
Theoretical study of retarded integral equations is almost
exclusively carried out in the Laplace domain. Invertibility of the
corresponding boundary integral equations is established in the
resolvent set of the Laplacian and bounds depending on the Laplace
transformed parameter are used to move back to the time domain.
While this proved to be an extremely powerful approach, leading to
solid estimates and useful in the analysis of CQ-type time
discretization, it is also clear that some properties are lost in
the process of transforming back and forth to the Laplace domain and
the some of the functional spaces that this kind of analysis impose
might not be optimal. This paper offers a complementary point of
view using elementary techniques from the theory of evolution
equations based on strongly continuous semigroups (cf.
\cite{EngNag06}, \cite{Kes89}). The main idea is recasting the
semidiscrete problem as a Cauchy problem for an equation of the
second order associated to an unbounded operator. Once the problem
is reexpressed as a second order equation in time, we use basic
estimates in the simplest possible setting (Hilbert space theory and
a coincidence of two of the four possible spaces that appear in the
abstract formulation) in order to simply derive the energy
conservation property. The kind of unbounded operator that will
appear in this formulation is very closely related to the exotic
transmission problems that Antonio Laliena and the author of this
paper devised for the analysis of CQ schemes applied to non-trivial
acoustic scattering problems in \cite{LalSay09}. The surprisingly
simple (and quite effective) idea consists of dealing with the
different aspects of Galerkin discretization (the fact that the
unknown is sought in a discrete space and the fact that the
corresponding equation is tested with the same space) as
non-standard transmission conditions.

The paper is organized as follows. In Section \ref{sec:1} we present
the indirect indirect formulation of the sound-soft scattering
problem and its Galerkin semidiscretization, and we state the energy
conservation result for this problem. The proof of this result is
carried out in Section \ref{sec:2} by using the techniques that have
been described before, and taking advantage of well known results
for the single layer potential associated to the Yukawa
(reaction-difussion) equation. In Section \ref{sec:3} we present a
boundary integral absorbing boundary condition for a wave
propagation problem in free space and its Galerkin (BEM-FEM)
discretization, and we state the corresponding energy estimate. The
proof of this result is the content of Section \ref{sec:4}. Section
\ref{sec:5} presents two easy generalizations: indirect formulation
for sound-hard scattering and a different absorbing boundary
condition.  In Section \ref{sec:6} we give two partially negative
results, showing that space discretization does not lead to energy
conservation either in a direct formulation of a scattering problem
or a one-equation absorbing boundary condition. Finally, Section
\ref{sec:A} includes the abstract frame of Cauchy problems for
equations of the second order associated to a class of unbounded
operators. These results are easy modifications of theorems that are
well known in the literature of $C_0-$semigroups and are presented
here in order to have an easy reference in the same language that we
are using in the paper.

\paragraph*{Notational foreword.}
Basic elementary results on Sobolev spaces will be assumed
throughout without specific reference. All of them can be found in
any text on the subject (\cite{AdaFou03} for instance). Given an
open set $U$ with Lipschitz boundary, we will consider the spaces
$L^2(U)$ and $H^m(U)$ for $m\ge 1$. The $L^2(U)-$norm will be
denoted $\| \punto\|_U$ and the $H^1(U)-$norm $\|\punto\|_{1,U}$.
Two fractional Sobolev spaces will be used on the boundary
$H^{\pm1/2}(\partial U)$. The space $H^1_0(U)$ is the kernel of the
trace operator $H^1(U)\to H^{1/2}(\partial U)$. The characteristic
function of $U$ will be denoted $\chi_U$.

Although the geometric layout of the different parts of this article
will vary, in all cases, there will be a bounded open set $\Omega$,
with exterior $\Omega^+:=\mathbb R^d\setminus\overline\Omega$ and
common Lipschitz boundary $\Gamma:=\partial\Omega=\partial\Omega^+$.
There are two possible traces on $\Gamma$, which will be
respectively denoted $\gamma^+:H^1(\Omega^+)\to H^{1/2}(\Gamma)$ and
$\gamma^-:H^1(\Omega) \to H^{1/2}(\Gamma)$. Whenever there is a
single trace or functions are exclusively defined in the interior
domain, the superscript will be dropped.

For functions of space and time variables $u(\mathbf x,t)$, we will
often employ the notation of theory of evolution equations, where
only the time variable is displayed. This amounts to considering
functions $u:[0,T]\to X$, where $X$ is a space of functions of the
$\mathbf x$ variable. With this notation, $\dot u$ and $\ddot u$
denote the first and second $t-$derivatives of $u(t)$.

\section{Single layer potentials for sound-soft
scattering}\label{sec:1}

Consider a bounded open set $\Omega\subset \mathbb R^d$, with
Lipschitz boundary $\Gamma:=\partial\Omega$ such that
$\Omega^+:=\mathbb R^d\setminus\overline\Omega$ is connected
($\Omega$ does not need to be connected though). Let $u_0$ and $v_0$
be given functions such that $\mathcal O:=\mathrm{supp}\, u_0 \cup
\mathrm{supp}\, v_0$ is compact and $\mathcal O\cap
\overline{\Omega}=\emptyset.$ We consider the following problem of
scattering of acoustic waves by a sound-soft obstacle: we look for
$u:[0,\infty) \to H^1(\Omega^+)$ such that for all $t>0$
\begin{subequations}\label{1.eq1}
\begin{alignat}{4}
\ddot u = \Delta u & & \qquad & \mbox{in $\Omega^+ $ },\label{1.eq1.a}\\
u=0 & & & \mbox{on $\Gamma $},\label{1.eq1.b}
\end{alignat}
\end{subequations}
and the initial conditions (on $\Omega^+$)
\begin{equation}\label{1.eq1b}
u(0)=u_0, \qquad \dot u(0)=v_0
\end{equation}
are satisfied. We will additionally assume that $u_0 \in
H^2_0(\Omega^+)$ and $v_0 \in H^1_0(\Omega^+)$. Let $R>0$ be such
that
\begin{equation}\label{1.eq11}
\overline\Omega \cup\mathcal O \subset B(0;R):=\{ \mathbf x \in
\mathbb R^d\,:\, |\mathbf x|< R\}.
\end{equation}
Because of the finite speed of propagation of waves, solutions of
problem \eqref{1.eq1}-\eqref{1.eq1b} are supported in $B(0;R+T)$ for
all $t\in [0,T]$.

The solution of \eqref{1.eq1}-\eqref{1.eq1b} will be decomposed as
the sum of a free (or incident) wave and a scattered wave. To define
the free wave we need to extend the initial data by zero to the
interior of the scatterer:
\[
\mathrm E v := \left\{ \begin{array}{ll} v \quad & \mbox{in
$\Omega^+$},\\ 0 & \mbox{in $\Omega$}.
\end{array}\right.
\]
The free wave is the solution of the wave equation in $\mathbb R^d$
for $t>0$
\begin{equation}\label{1.eq3}
\ddot u^{\mathrm{free}} = \Delta u^{\mathrm{free}}
\end{equation}
satisfying initial conditions
\begin{equation}\label{1.eq3a}
u^{\mathrm{free}}(0)=\mathrm E u_0, \qquad \dot
u^{\mathrm{free}}(0)=\mathrm Ev_0.
\end{equation}
The solution to this problem satisfies
\begin{equation}\label{1.eq3b}
u^{\mathrm{free}}\in \mathcal C^2([0,\infty); L^2(\mathbb R^d))\cap
\mathcal C^1 ([0,\infty); H^1(\mathbb R^d))\cap \mathcal
C([0,\infty), H^2(\mathbb R^d))
\end{equation}
and $\mathrm{supp}\, u^{\mathrm{free}}(t) \subset B(0;R+t)$ for all
$t$. The scattered wave is the difference $ u-u^{\mathrm{free}}$.

\begin{remark}\rm
For the kind of arguments that we are going to develop, we will
assume that the free wave is known. Section \ref{sec:3} deals with a
semidiscrete version of a problem that generalizes
\eqref{1.eq3}-\eqref{1.eq3a}, leading to algorithms to compute the
propagation of compactly supported initial data in free space. In
applications, the free or incident wave is actually known. Typically
the free wave is the solution of a non--homogeneous wave equation
and has singularities at some points away from the scatterer. Other
practical incident waves include plane waves that are not compactly
supported at any given time and have infinite energy. Because our
interest lies in how discretization affects the integral model used
to represent $u-u^{\mathrm{free}}$, the use of free waves as those
of \eqref{1.eq3}-\eqref{1.eq3a} fulfills our needs.
\end{remark}

The scattered wave $u-u^{\mathrm{free}}$ can be represented with a
single layer retarded potential \cite{BamHaD86a}. For $(\mathbf
x,t)\in \Omega^+\times[0,\infty)$ and a given density
$\psi:[0,\infty) \to H^{-1/2}(\Gamma)$ the single layer potential is
defined as
\begin{equation}\label{1.eq4}
(\mathcal S* \psi)(\mathbf x,t) = \left\{
\begin{array}{ll} \ds \int_\Gamma \frac{\psi(\mathbf y,t-|\mathbf
x-\mathbf y|)}{4\pi |\mathbf x-\mathbf y|}\mathrm d\Gamma(\mathbf y)
 & \mbox{(when $d=3$),}\\
\ds\int_\Gamma \int_0^{t-|\mathbf x-\mathbf
y|}\hspace{-15pt}\frac{\psi(\mathbf
y,\tau)}{\sqrt{(t-\tau)^2-|\mathbf x-\mathbf y|^2}}\mathrm
d\Gamma(\mathbf y)\mathrm d\tau & \mbox{(when $d=2$).}
\end{array}\right.
\end{equation}
The integral forms in \eqref{1.eq4} are only valid for smooth
densities. Weak forms of the potentials have to be used in general
(see \cite{LalSay09b} for full justification in the three
dimensional case). We will keep the convolutional notation  for the
layer potential $\mathcal S*\psi$ in order to distinguish time
domain potentials and operators from similar entities for steady
state problems. Let
\[
\mathcal V*\psi :=\gamma (\mathcal S*\psi)
\]
be the corresponding single layer integral operator (see
\cite{BamHaD86a} and \cite{HaD03}). The indirect representation of
the scattered field looks for a density $\psi:\mathbb R \to
H^{-1/2}(\Gamma)$ such that
\begin{equation}\label{1.eq5}
\psi\equiv 0 \mbox{ in $(-\infty,0)$}\qquad \mathcal V*\psi+\gamma
u^{\mathrm{free}}=0
\end{equation}
(this is an equation on $\Gamma\times [0,\infty)$) and then
represents the total wave by
\begin{equation}\label{1.eq6}
u=\mathcal S*\psi+u^{\mathrm{free}}.
\end{equation}
Equation \eqref{1.eq5} is now approximated with a Galerkin method
{\em only in the space variable}. To do that, we choose a sequence
of finite dimensional spaces
\[
X_h \subset H^{-1/2}(\Gamma),
\]
approximate \eqref{1.eq5} by the problem
\begin{equation}\label{1.eq7}
\left[\begin{array}{l} \psi_h:\mathbb R \to X_h,\qquad \psi_h \equiv 0 \mbox{ in $(-\infty,0)$},\\[1.5ex]
\langle \mu_h,\mathcal V*\psi_h+\gamma
u^{\mathrm{free}}\rangle_\Gamma =0 \qquad \forall\mu_h \in X_h \quad
\forall t,
\end{array}\right.
\end{equation}
and write an approximation of the total wave as
\begin{equation}\label{1.eq8}
u_h:=\mathcal S*\psi_h+u^{\mathrm{free}}.
\end{equation}
The angled bracket in \eqref{1.eq7} denotes the $H^{-1/2}(\Gamma)
\times H^{1/2}(\Gamma)$ duality product. Note that
$\mathrm{supp}\,u_h(t)\subset B(\mathbf 0;R+t)$ for all $t$, because
the layer potential $\mathcal S*\psi$ propagates a wave from
$\Gamma$ at unit velocity starting at time $t=0$.

\begin{remark}\rm Problem \eqref{1.eq7} is a system of functional
equations. Let us assume that $X_h \subset L^\infty(\Gamma)$ and
that $\{ N_j\,:\, j=1,\ldots m\}$ is a basis of $X_h$. We can then
write the unknown density as
\[
\psi_h(\mathbf y,t)=\sum_{j=1}^M \psi_j(t) N_j(\mathbf x) \qquad
\psi_j:\mathbb R \to \mathbb R, \qquad \mathrm{supp}\,\psi_j \subset
(0,\infty).
\]
In the three dimensional case, \eqref{1.eq7} is then equivalent to
\begin{eqnarray}\nonumber
\sum_{j=1}^M \int_\Gamma \int_\Gamma \frac{N_j(\mathbf y)N_i(\mathbf
x)}{4\pi |\mathbf x-\mathbf y|} \psi_j(t-|\mathbf x-\mathbf y|)
\mathrm d\Gamma(\mathbf x)\mathrm d\Gamma(\mathbf y)=-\int_\Gamma
N_i(\mathbf x) u^{\mathrm{free}}(\mathbf x,t)\mathrm d\Gamma(\mathbf
x) & &\\
 i=1,\ldots,M.\label{1.eq20}
\end{eqnarray}
This system of functional equations include integrated delays of all
the unknowns.
\end{remark}

We are now in conditions to state the main theorem of this part of
the article, showing energy conservation and regularity in time for
the semidiscrete total wave field. The proof of this theorem will be
given in Section \ref{sec:2}.

\begin{theorem}\label{1.th1}
The semidiscrete total field $u_h$ and the associated density
$\psi_h$, given by \eqref{1.eq7}-\eqref{1.eq8}, satisfy
\begin{eqnarray}
\label{1.eq9b} u_h &\in& \mathcal C^2([0,\infty); L^2(\mathbb
R^d))\cap \mathcal C^1([0,\infty);H^1(\mathbb R^d)),\\
\label{1.eq9c} \psi_h &\in& \mathcal C([0,\infty);
H^{-1/2}(\Gamma)).
\end{eqnarray}
Moreover, the energy
\begin{equation}\label{1.eq9}
\smallfrac12 \| \nabla u_h(t)\|_{\mathbb R^d}^2 +\smallfrac12 \|\dot
u_h(t)\|_{\mathbb R^d}^2
\end{equation}
is constant over time.
\end{theorem}

Let us emphasize that the exact total wave field $u$ is only defined
in $\Omega^+$ and that it can be extended by zero to the interior of
the obstacle. The semidiscrete total wave field $u_h$ is however
defined in $\mathbb R^d$ and the property of energy conservation is
proved with integration over $\mathbb R^d$ and not only over
$\Omega^+$. This shows that Galerkin discretization leaks part of
the energy to the interior of the obstacle, although the total
energy is still constant.

\section{Proof of Theorem \ref{1.th1}}\label{sec:2}

\paragraph{Notation.}  The jump of the trace across $\Gamma$ is denoted
\[
\jump{\gamma u}:=\gamma^-u-\gamma^+ u.
\]
Weak normal derivatives on $\Gamma$ can be defined using Green's
formula. Given an open bounded set $\mathbb B$ that contains
$\overline\Omega$, we can define $\partial_\nu^\pm u$ for any $u \in
H^1(\mathbb B\setminus\Gamma)$ such that $\Delta u \in L^2(\mathbb
B\setminus\Gamma)$ with the formulas:
\begin{eqnarray*}
\langle\partial_\nu^-u,\gamma^- v\rangle_\Gamma &=& (\nabla u,\nabla
v)_{\Omega^-}+(\Delta u,v)_{\Omega^-}\qquad \forall v \in
H^1(\Omega^-),\\
\langle\partial_\nu^+u,\gamma^+ v\rangle_\Gamma &=& -(\nabla
u,\nabla v)_{\mathbb B\cap \Omega^+}-(\Delta u,v)_{\mathbb
B\cap\Omega^+} \qquad \forall v\in H^1_{\partial\mathbb B}(\mathbb
B\cap \Omega^+),
\end{eqnarray*}
where
\begin{equation}\label{2.eq21}
H^1_{\partial\mathbb B}(\mathbb B\cap \Omega^+)=\{ v \in H^1(\mathbb
B\cap\Omega^+)\,:\, \gamma_{\partial\mathbb B}v=0\}
\end{equation}
and $\gamma_{\partial\mathbb B}$ is the trace operator associated to
the boundary of $\mathbb B$. It is well known that the definition of
the exterior normal derivative is independent of the set $\mathbb
B$. The jump of the normal derivative $\jump{\partial_\nu
u}:=\partial_\nu^-u-\partial_\nu^+ u$ is defined for the same class
of functions.

\paragraph{The single layer Yukawa potential.}
Consider the fundamental solution of the Yukawa operator $u\mapsto
-\Delta u+u$:
\begin{equation}\label{2.eq23}
E(\mathbf x,\mathbf y):= \left\{\begin{array}{ll} \ds
\frac{e^{-|\mathbf x-\mathbf y|}}{4\pi |\mathbf x-\mathbf y|} &
\mbox{(when $d=3$),}\\[1.5ex]
\ds\frac1{2\pi}K_0(|\mathbf x-\mathbf y|)& \mbox{(when $d=2$)},
\end{array}\right.
\end{equation}
where $K_0$ is the modified Bessel function of the second kind and
order zero. On the surface/curve $\Xi:=\partial\mathbb B\cup
\Gamma$, we can define the single layer potential
\begin{equation}\label{2.eq25}
\mathrm S  \lambda := \int_\Xi E(\punto,\mathbf y)\lambda(\mathbf
y)\mathrm d\Xi(\mathbf y).
\end{equation}
Using a weak definition of this potential (see \cite{Cos88} or the
very general theory developed in \cite{McL00}), we can prove that
$\mathrm S :H^{-1/2}(\Xi) \to H^1(\mathbb R^d)$ is bounded,
that
\begin{equation}\label{2.eq24}
\mathrm V :=\gamma_\Xi \mathrm S:H^{-1/2}(\Xi)\to H^{1/2}(\Xi)
\end{equation}
is bounded and coercive
\begin{equation}\label{2.eq5}
\langle \lambda, \mathrm V \lambda\rangle_\Xi \ge C \|
\mu\|_{-1/2,\Xi}^2 \qquad \forall \lambda \in H^{-1/2}(\Xi),
\end{equation}
and that $\Delta (\mathrm S\lambda)=\mathrm S\lambda$ in $\mathbb
R^d\setminus\Xi$ for all $\lambda$. Also, if we write
$\lambda=(\lambda_\Gamma,\lambda_{\partial})\in H^{-1/2}(\Xi)\cong
H^{-1/2}(\Gamma)\times H^{-1/2}(\partial\mathbb B)$, then
$\jump{\partial_\nu (\mathrm S\lambda)}=\lambda_\Gamma$ (the jump is
defined across $\Gamma$).

\paragraph{A transmission problem.} The single layer retarded potential
satisfies the following properties \cite{BamHaD86a}
\begin{equation}\label{2.eq1}
\jump{\gamma(\mathcal S*\psi)}=0 \qquad \jump{\partial_\nu(\mathcal
S*\psi)}=\psi \qquad (\mathcal S*\psi)(0)=0 \qquad
\smallfrac{d}{dt}(\mathcal S*\psi)(0)=0.
\end{equation}
For given initial data with compact support we can choose $R>0$ so
that \eqref{1.eq11} is satisfied. Since the speed of propagation of
$u^{\mathrm{free}}$ and $\mathcal S*\psi_h$ is the same,
\begin{equation}\label{2.eq2}
\mathrm{supp}\, u_h(t) \subset \mathbb B:= B(\mathbf 0;R+T) \qquad
\forall t \in [0,T].
\end{equation}
Noticing that
\begin{equation}\label{2.eq3}
\jump{\gamma u_h}= \jump{\gamma(\mathcal S*\psi_h)}+\jump{\gamma
u^{\mathrm{free}}}=0,
\end{equation}
it follows that, restricted to the time interval $[0,T]$, the
function defined by \eqref{1.eq7}-\eqref{1.eq8} can be understood as
$u_h:[0,T]\to H^1(\mathbb B\setminus\Gamma)$ that solves the wave
propagation problem:
\begin{subequations}\label{2.eq4}
\begin{alignat}{4}
\label{2.eq4a}
\ddot u_h =\Delta_{\pm} u_h, \\
\jump{\gamma u_h}=0, \\ \label{2.eq4d}
\gamma u_h \in X_h^\circ,\\
\label{2.eq4e} \jump{\partial_\nu u_h}\in X_h,\\
\gamma_{\partial\mathbb B}u_h=0,
\end{alignat}
\end{subequations}
with initial conditions
\begin{equation}\label{2.eq4comp}
u_h(0)=\mathrm Eu_0 \qquad \dot u_h(0)=\mathrm E v_0.
\end{equation}
The set $X_h^\circ$ in \eqref{2.eq4d} is the polar set of $X_h$,
i.e.,
\[
X_h^\circ:=\{ \xi \in H^{1/2}(\Gamma)\,:\,
\langle\mu_h,\xi\rangle_\Gamma =0 \quad \forall \mu_h \in X_h\}.
\]
The Laplace operator $\Delta_\pm$ in \eqref{2.eq4a} is the one
defined in the sense of distributions in $\mathbb B\setminus\Gamma$.
Finally, the transmission condition \eqref{2.eq4e} is equivalent to
\[
\langle\jump{\partial_\nu u_h},\xi_h \rangle_\Gamma =0 \qquad
\forall \xi_h \in X_h^\circ.
\]
Let now $u_h$ be a solution of \eqref{2.eq4}-\eqref{2.eq4comp} and
define $\psi_h:=\jump{\partial_\nu u_h}=\jump{\partial_\nu
(u_h-u^{\mathrm{free}})}$ (recall \eqref{1.eq3b}). We can then show
that $\mathcal S*\psi_h=u_h-u^{\mathrm{free}}$ by comparing the
transmission problems that both solutions satisfy.

\paragraph{Formulation as a Cauchy problem.} Consider the Hilbert
spaces
\begin{eqnarray*}
H &:=& L^2(\mathbb B),\\
V &:=& \{ u \in H^1_0(\mathbb B)\,:\,\gamma u \in X_h^\circ\},\\
D(A) &:=& \{ u \in V\,:\, \Delta_\pm u\in L^2(\mathbb B), \quad
\jump{\partial_\nu u}\in X_h\},
\end{eqnarray*}
endowed with the respective norms
\[
\| u\|_H:= \| u\|_{\mathbb B} \qquad \| u\|_V:= \| \nabla
u\|_{\mathbb B}, \qquad \| u\|_{D(A)}:=\left(\| \nabla u\|_{\mathbb
B}^2+\|\Delta_\pm u\|_{\mathbb B}^2\right)^{1/2}.
\]
We also consider the operator $A:=\Delta_\pm$. We next verify the
two conditions of Section \ref{sec:A}. First of all, the generalized
Green's Identity: for $u \in D(A)$, $v \in V\subset H^1_0(\mathbb
B)$, using the weak definition of the normal derivatives, it follows
that
\[
(\nabla u,\nabla v)_{\mathbb B}+(\Delta_\pm u,v)_{\mathbb B}=\langle
\jump{\partial_\nu u},\gamma v\rangle_\Gamma =0,
\]
because $\jump{\partial_\nu u}\in X_h$ and $\gamma v \in X_h^\circ$.
The second step is surjectivity: for any $f \in L^2(\mathbb B)$ we
want to find
\begin{equation}\label{2.eq7}
u \in D(A), \qquad -\Delta u+u=f \quad \mbox{in $\mathbb
B\setminus\Gamma$}.
\end{equation}
We first choose $u^{\mathrm{nh}}\in H^1_0(\mathbb B)$ such that
$-\Delta u^{\mathrm{nh}}+ u^{\mathrm{nh}}=f$ in $\mathbb B$, and
note that $u^{\mathrm{nh}}\in H^2(\mathbb B)$ by a simple regularity
argument. We next consider a variational problem in the space
$\underline X_h:=X_h \times H^{-1/2}(\partial\mathbb B) \subset
H^{-1/2}(\Xi)$:
\begin{equation}\label{2.eq6}
\left[\begin{array}{l} \lambda=(\lambda_h,\lambda_{\partial})\in
\underline X_h,\\[1.5ex]
\langle \rho,\mathrm V \lambda\rangle_{\Xi}=-\langle \rho_h,\gamma
u^{\mathrm{nh}}\rangle_\Gamma \qquad \forall
\rho=(\rho_h,\rho_\partial) \in \underline X_h.
\end{array}
\right.
\end{equation}
This problem is uniquely solvable by the coercivity property
\eqref{2.eq5}. We finally take $u:= u^{\mathrm{nh}}+\mathrm
S\lambda.$ It is clear that $u\in H^1(\mathbb B)$. Testing
\eqref{2.eq6} with elements $(0,\rho_\partial)\in \{ 0\}\times
H^{-1/2}(\partial\mathbb B)$ and recalling that $\gamma_\Xi \mathrm
S =\mathrm V$ it follows that $u \in H^1_0(\mathbb B)$. Testing with
elements $(\rho_h,0)\in X_h \times \{0\}$ it follows that $\gamma u
\in X_h^\circ$. This proves that $u \in V$. Also, $-\Delta u+u=f$ in
$\mathbb B\setminus\Gamma$ and $\jump{\partial_\nu
u}=\jump{\partial_\nu \mathrm S \mu}=\lambda_h \in X_h$. Therefore
$u \in D(A)$ and \eqref{2.eq7} is satisfied.

\paragraph{Conclusions.} The theory for abstract Cauchy problems
that is sketched in Section \ref{sec:A} can be applied to problem
\eqref{2.eq4}. When initial data are in $D(A)\times V$, the Cauchy
problem has a unique strong $\mathcal C^2$ solution. Since $u_0\in
H^2(\Omega^+)$ and $v_0\in H^1(\Omega^+)$ vanish in a neighborhood
of $\overline\Omega$ and also in a neighborhood of $\partial\mathbb
B$, then $(\mathrm E u_0,\mathrm E v_0)$ satisfies the regularity
requirements to have strong solutions:
\[
u_h \in \mathcal C^2([0,T];L^2(\mathbb B))\cap \mathcal
C^1([0,T];H^1_0(\mathbb B))\cap \mathcal C([0,T]; D(A)).
\]
Extending by zero to the exterior of $\mathbb B$, it is clear that
\eqref{1.eq9} is satisfied. Note that the property is proved in
$[0,T]$ for any $T$ and that $T$ influences the size of the ball
$\mathbb B$ that we use to truncate the domain. Note also that
$\jump{\partial_\nu\punto}:D(A) \to H^{-1/2}(\Gamma)$ is bounded.
Then $\psi_h:=\jump{\partial_\nu u_h}\in \mathcal C([0,\infty);
H^{-1/2}(\Gamma))$.

\begin{remark}\rm The convenience of adding the cut-off boundary
$\partial\mathbb B$ --far enough from the obstacle so that the
solution is not affected in the time interval $[0,T]$-- is due to
the difficulty of working with energy spaces in unbounded domains
(see Section \ref{sec:A}).
\end{remark}

\section{Transparent boundary conditions}\label{sec:3}

We now consider the problem of propagation of compactly supported
initial conditions in free space $\mathbb R^d$
\[
c^{-2} \ddot u = \nabla \cdot (\kappa\nabla u)
\]
with initial conditions
\[
u(0)=u_0 ,\qquad \dot u(0)=v_0.
\]
We assume that $\kappa, c \in L^\infty(\mathbb R^d)$ are positive
and that $\kappa^{-1}, c^{-1}\in L^\infty(\mathbb R^d)$.
Furthermore, we assume that $c\equiv 1 $ and $\kappa\equiv 1$
outside a compact set. Let
\[
\mathcal O:=\mathrm{supp}\,u_0 \cup \mathrm{supp}\,v_0\cup
\mathrm{supp}\,(c-1) \cup \mathrm{supp}\,(\kappa-1),
\]
which is a compact set by all the previous hypotheses on initial
data and coefficients. We choose a bounded open set $\Omega$ with
Lipschitz boundary $\Gamma:=\partial\Omega$, with connected exterior
$\Omega^+:=\mathbb R^d\setminus\Gamma$ and such that $\mathcal
O\subset \overline\Omega$. We admit the possibility of taking
$\overline\Omega=\mathcal O$, as long as $\mathcal O$ meets the
regularity hypotheses required for $\Omega$. Also, there is no need
to have $\Omega$ connected, although for simplicity we will assume
that its exterior is connected.

Denoting the conormal derivative on $\Gamma$ by $\partial_\nu^\kappa
u=(\kappa\nabla u)\cdot\boldsymbol\nu$, and renaming
$u^{\mathrm{ext}}=u|_{\Omega^+}$, we have the problem
\begin{subequations}\label{3.eq1}
\begin{alignat}{4}
\label{3.eq1a}
c^{-2} \ddot u = \nabla \cdot (\kappa\nabla u) & &
\qquad & \mbox{in
$\Omega$ },\\
\ddot u^{\mathrm{ext}}=\Delta u^{\mathrm{ext}} & & & \mbox{in $\Omega^+$,}\\
\label{3.eq1c}
\gamma u=\gamma^+ u^{\mathrm{ext}} & & & \mbox{on $\Gamma$,}\\
\label{3.eq1d}
\partial_\nu^\kappa u=\partial_\nu u^{\mathrm{ext}}& & & \mbox{on
$\Gamma$,},
\end{alignat}
\end{subequations}
with initial conditions
\begin{equation}\label{3.eq1comp}
u(0)=u_0, \quad \dot u(0)=v_0, \quad u^{\mathrm{ext}}(0)=0, \quad
\dot u^{\mathrm{ext}}(0)=0.
\end{equation}
The function $u^{\mathrm{ext}}$ can be represented with Kirchhoff's
formula
\begin{equation}\label{3.eq2}
u^{\mathrm{ext}}=\mathcal D*\varphi -\mathcal S*\lambda \qquad
(\varphi,\lambda):=(\gamma^+
u^{\mathrm{ext}},\partial_\nu^+u^{\mathrm{ext}}),
\end{equation}
using both Cauchy data on the interface $\Gamma$. This formula
employs the double layer retarded potential (see \cite{BamHaD86b},
\cite{HaD03})
\[
(\mathcal D*\varphi) (\mathbf x,t) = \left\{\begin{array}{ll}\ds
\int_\Gamma \nabla_{\mathbf y}\left(\frac{\varphi(\mathbf
z,t-|\mathbf x-\mathbf y|)}{4\pi|\mathbf x-\mathbf
y|}\right)\Big|_{\mathbf z=\mathbf y}\cdot\boldsymbol\nu(\mathbf
y)\mathrm d\Gamma(\mathbf y) & \mbox{(when
$d=3$)}\\[2.5ex]
\ds \int_\Gamma \int_0^{t-|\mathbf x-\mathbf y|}
\hspace{-15pt}\frac{\varphi(\mathbf y,\tau)}{|\mathbf x-\mathbf
y|^2-(t-\tau)^2}\Phi(\mathbf x,\mathbf y,t-\tau)\mathrm
d\Gamma(\mathbf
y)\mathrm d\tau\\
\ds \hspace{1cm}+\int_\Gamma \frac{\varphi(\mathbf y,t-|\mathbf
x-\mathbf y|)}{|\mathbf x-\mathbf y|}\Phi(\mathbf x,\mathbf
y,t-\tau)\mathrm d\Gamma(\mathbf y) &  \mbox{(when $d=2$)},
\end{array}\right.
\]
where
\[
\Phi(\mathbf x,\mathbf y,t)=\frac{(\mathbf x-\mathbf
y)\cdot\boldsymbol\nu(\mathbf y)}{\sqrt{t^2-|\mathbf x-\mathbf
y|^2}}
\]
and $\boldsymbol\nu(\mathbf y)$ is the outwards pointing normal
vector at $\mathbf y$. Three retarded integral operators appear in
this formulation:
\begin{equation}\label{3.eq3}
\mathcal K^t*\lambda:=
\smallfrac12(\partial_\nu^++\partial_\nu^-)(\mathcal S*\lambda)
\qquad \mathcal K*\varphi=\smallfrac12(\gamma^++\gamma^-)(\mathcal
D*\varphi) \qquad \mathcal W*\varphi=-\partial_\nu^\pm(\mathcal
D*\varphi).
\end{equation}
The boundary-field formulation is based on representing
$u^{\mathrm{ext}}$ with the formula \eqref{3.eq2}. We then write a
weak formulation for the interior equation \eqref{3.eq1a} and
substitute the transmission condition \eqref{3.eq1d}:
\begin{equation}\label{3.eq4}
(c^{-2} \ddot u,v)_\Omega+(\kappa\nabla u,\nabla
v)_\Omega-\langle\lambda,\gamma v\rangle_\Gamma =0 \qquad \forall v
\in H^1(\Omega).
\end{equation}
A second equation is obtained by imposing the transmission condition
\eqref{3.eq1c}, with $u^{\mathrm{ext}}$ written in terms of its
Cauchy data \eqref{3.eq2} and using the conditions \eqref{3.eq3} to
represent the trace of $\mathcal D*\varphi$:
\begin{equation}\label{3.eq5}
\gamma u+\mathcal V*\lambda-(\smallfrac12\varphi+\mathcal
K*\varphi)=0.
\end{equation}
The third equation is an identity satisfied by the Cauchy data:
\begin{equation}\label{3.eq6}
\smallfrac12\lambda+\mathcal K^t*\lambda+\mathcal W*\varphi=0.
\end{equation}
Finally, the global formulation results from writing
\eqref{3.eq4}-\eqref{3.eq5}-\eqref{3.eq6}. The unknowns are
\[
u:[0,\infty) \to H^1(\Omega), \qquad \lambda:\mathbb R \to
H^{-1/2}(\Gamma), \qquad \varphi:\mathbb R \to H^{1/2}(\Gamma),
\]
with initial conditions
\[
u(0)=u_0, \qquad \dot u(0)=v_0, \qquad \lambda\equiv 0 \mbox{ in
$(-\infty,0)$}, \qquad \varphi\equiv 0\mbox{ in $(-\infty,0)$}.
\]
We now choose spaces
\[
V_h \subset H^1(\Omega) \qquad X_h \subset H^{-1/2}(\Gamma) \qquad
Y_h \subset H^{1/2}(\Gamma),
\]
and look for
\[
u_h:[0,\infty)\to V_h, \qquad \lambda_h:\mathbb R \to X_h, \qquad
\varphi_h:\mathbb R\to Y_h,
\]
satisfying initial conditions
\begin{equation}\label{3.eq7}
u_h(0)=u_{h,0}, \qquad \dot u_h(0)=v_{h,0}, \qquad \lambda_h\equiv 0
\mbox{ in $(-\infty,0)$}, \qquad \varphi_h\equiv 0\mbox{ in
$(-\infty,0)$},
\end{equation}
for approximations $u_0\approx u_{h,0}\in V_h$ and $v_0\approx
v_{h,0}\in V_h$ to be determined by a projection method or by some
kind of interpolation process (this is not relevant in the sequel).
Finally, we have the set of Galerkin equations: for $t> 0$
\begin{subequations}\label{3.eq8}
\begin{alignat}{4}
\label{3.eq8a} (c^{-2} \ddot u_h,v_h)_\Omega+(\kappa\nabla
u_h,\nabla v_h)_\Omega-\langle\lambda_h,\gamma v_h\rangle_\Gamma =0
& & \qquad&
\forall v_h \in V_h ,\\
\label{3.eq8b} \langle\mu_h,\gamma u_h\rangle_\Gamma +
\langle\mu_h,\mathcal
V*\lambda_h\rangle_\Gamma-\langle\mu_h,\smallfrac12\varphi_h+\mathcal
K*\varphi_h\rangle_\Gamma=0 & & & \forall \mu_h \in X_h,\\
\label{3.eq8c} \langle\smallfrac12\lambda_h+\mathcal
K^t*\lambda_h,\xi_h\rangle_\Gamma +\langle \mathcal
W*\varphi_h,\xi_h\rangle_\Gamma=0 & & & \forall\xi_h \in Y_h.
\end{alignat}
\end{subequations}
Equations \eqref{3.eq8} form a system of linear second order
differential equations coupled with the kind of retarded equations
that we found in Section \ref{sec:2} (see \eqref{1.eq20} for
instance). These equations are complemented with the initial
conditions \eqref{3.eq7}. Once \eqref{3.eq7}-\eqref{3.eq8} has been
solved (its solvability is part of what we state in the next
theorem), we can define the approximation to the exterior solution
\begin{equation}\label{3.eq9}
u_h^{\mathrm{ext}}:= \mathcal D*\varphi_h-\mathcal S*\lambda_h.
\end{equation}
Note that this function is defined in $\mathbb R^d$ and not only in
$\Omega^+$. In order to simplify some of the forthcoming arguments,
we will assume that constant functions belong to the three discrete
spaces
\[
\mathbb P_0(\Omega)\subset V_h \qquad \mathbb P_0(\Gamma) \subset
X_h \qquad \mathbb P_0(\Gamma)\subset Y_h.
\]
When $\Gamma$ and $\Omega$ are not connected, the spaces $\mathbb
P_0(\Gamma)$ and $\mathbb P_0(\Omega)$ have to be understood as the
spaces of constant functions on each connected component of the
corresponding domain.

The following result (which will be proved in Section \ref{sec:4})
gives a basic regularity estimate for the solution of this problem.
It also states an energy conservation property, where in addition to
the expected wave fields ($u_h$ in the interior domain $\Omega$ and
$u_h^{\mathrm{ext}}$ in the exterior domain $\Omega^+$), we have to
count the energy of $u_h^{\mathrm{ext}}$ in the interior of
$\Gamma$, computed with the material properties of the surrounding
medium.

\begin{theorem}\label{3.th1}
The semidiscrete total fields $(u_h,u_h^{\mathrm{ext}})$ and the
approximations to the Cauchy data on the interface
$(\varphi_h,\lambda_h)$ given by
\eqref{3.eq7}-\eqref{3.eq8}-\eqref{3.eq9} satisfy
\begin{eqnarray}
u_h & \in & \mathcal C^2([0,\infty); L^2(\Omega))\cap \mathcal
C^1([0,\infty);H^1(\Omega)),\\
u_h^{\mathrm{ext}} & \in & \mathcal C^2([0,\infty);L^2(\mathbb R^d)
\cap \mathcal C^1([0,\infty); H^1(\mathbb R^d\setminus\Gamma)),\\
\lambda_h &\in & \mathcal C([0,\infty);H^{-1/2}(\Gamma)),\\
\varphi_h &\in & \mathcal C^1([0,\infty);H^{1/2}(\Gamma)).
\end{eqnarray}
Moreover, the energy
\begin{equation}
\smallfrac12 \|\kappa^{1/2}\nabla u_h(t)\|_\Omega^2+\smallfrac12 \|
\nabla u_h^{\mathrm{ext}}(t)\|_{\mathbb R^d\setminus\Gamma}^2+
\smallfrac12\| c^{-1} \dot u_h(t)\|_\Omega^2+\smallfrac12\|\dot
u_h^{\mathrm{ext}}(t)\|_{\mathbb R^d}
\end{equation}
is constant over time.
\end{theorem}

\section{Proof of Theorem \ref{3.th1}}\label{sec:4}

\paragraph{Introduction of a cut-off boundary.} We will prove the
theorem for an arbitrary interval $[0,T]$. Since
$u_h^{\mathrm{ext}}$ is defined with retarded potentials whose
densities are causal functions (see the initial conditions
\eqref{3.eq7}), we can pick a sufficiently large radius $R>0$ so
that
\begin{equation}\label{4.eq1}
\mathrm{supp}\,u_h^{\mathrm{ext}}(t) \subset \mathbb B:=B(\mathbf 0;
R) \qquad \forall t\in [0,T].
\end{equation}

\paragraph{A transmission problem.} The first step towards the proof
consists of writing \eqref{3.eq8} in terms of the fields $(u_h,
u_h^{\mathrm{ext}})$. The field $u_h^{\mathrm{ext}}$  satisfies the
wave equation in $\mathbb R^d\setminus\Gamma$. The interior field
$u_h$ does not satisfy a differential equation though. Note that:
\begin{equation}\label{4.eq2}
\jump{\gamma u_h^{\mathrm{ext}}}= -\varphi_h \qquad
\jump{\partial_\nu u_h^{\mathrm{ext}}}=-\lambda_h
\end{equation}
The transmission problem looks for
\begin{equation}\label{4.eq3}
u_h:[0,\infty) \to V_h, \qquad u_h^{\mathrm{ext}}:[0,\infty) \to
H^1(\mathbb B\setminus\Gamma)
\end{equation}
satisfying initial conditions
\begin{equation}\label{4.eq4}
u_h(0)=u_{h,0},\qquad \dot u_h(0)=v_{h,0},\qquad
u_h^{\mathrm{ext}}(0)=0, \qquad \dot u_h^{\mathrm{ext}}(0)=0,
\end{equation}
and the equations for all $t>0$
\begin{subequations}\label{4.eq5}
\begin{alignat}{4}\label{4.eq5a}
(c^{-2}\ddot u_h,v_h)_\Omega+(\kappa\nabla u_h,\nabla v_h)+
\langle\jump{\partial_\nu u_h^{\mathrm{ext}}},\gamma
v_h\rangle_\Gamma =0 & & \qquad & \forall v_h \in V_h,\\
\label{4.eq5c}
\ddot u_h^{\mathrm{ext}}=\Delta_\pm u_h^{\mathrm{ext}}, & & & \\
\gamma_{\partial\mathbb B} u_h^{\mathrm{ext}}=0, & & &\\
\label{4.eq5d} \jump{\gamma u_h^{\mathrm{ext}}}\in Y_h, \qquad
\gamma u_h-\gamma^+ u_h^{\mathrm{ext}}\in X_h^\circ,\\
\label{4.eq5e} \jump{\partial_\nu u_h^{\mathrm{ext}}}\in X_h ,\qquad
\partial_\nu^- u_h^{\mathrm{ext}}\in Y_h.
\end{alignat}
\end{subequations}
Equation \eqref{4.eq5a} corresponds to \eqref{3.eq8a} after
substituting $\lambda_h=-\jump{\partial_\nu u_h^{\mathrm{ext}}}$
(see \eqref{4.eq2}. The exterior boundary condition \eqref{4.eq5c}
is a consequence of \eqref{4.eq1}. The first condition in both
\eqref{4.eq5d} and \eqref{4.eq5e} is a consequence of \eqref{4.eq2}.
The second condition in \eqref{4.eq5d} is just \eqref{3.eq8b}.
Finally, the second condition in \eqref{4.eq5e} is \eqref{3.eq8c}.

\paragraph{Formulation as a Cauchy problem.} Consider the space
\[
H^1_{\partial\mathbb B}(\mathbb B\setminus\Gamma):=\{ u\in
H^1(\mathbb B\setminus\Gamma)\,:\, \gamma_{\partial\mathbb B}u=0\}
\cong H^1(\Omega)\times H^1_{\partial\mathbb B}(\mathbb
B\cap\Omega^+)
\]
(recall \eqref{2.eq21}). The three required spaces to fit in the
frame of Section \ref{sec:A} are:
\begin{eqnarray}
H &:=& V_h \times L^2(\mathbb B)\\
V &:=& \{ (u_h,u^\star)\in V_h \times H^1_{\partial\mathbb
B}(\mathbb B\setminus\Gamma)\,:\, \jump{\gamma u^\star}\in Y_h \quad
\gamma u_h-\gamma^+ u^\star \in
X_h^\circ\}\\
D(A)&:=& \{ (u_h, u^\star)\in V\,:\,\Delta_\pm u^\star \in
L^2(\mathbb B)\, \quad \jump{\partial_\nu u^\star}\in X_h, \quad
\partial_\nu u^\star \in Y_h^\circ\}.
\end{eqnarray}
The norm of $H$ is
\[
\| (u_h,u^\star)\|_H^2:=\| c^{-1} u_h\|_\Omega^2+\|
u^\star\|_{\mathbb B}^2.
\]
The following norm
\[
H^1(\Omega)\times H^1_{\partial\mathbb B}(\mathbb
B\setminus\Gamma)\ni (u,u^\star)\mapsto \| \kappa^{1/2} \nabla
u\|_\Omega^2+\|\nabla u^\star\|_{\mathbb B\setminus\Gamma}^2 +
\left|\int_\Gamma (\gamma u-\gamma^+ u)\right|^2 +
\left|\frac1{|\Omega|}\int_\Omega u^\star\right|^2.
\]
can be easily shown to be equivalent to the usual Sobolev norm in
this space. This allows us to write a norm in $V$:
\[
\|(u_h,u^\star)\|_V^2:=\| \kappa^{1/2} \nabla
u_h\|_\Omega^2+\|\nabla u^\star\|_{\mathbb B\setminus\Gamma}^2 +
\left|\frac1{|\Omega|}\int_\Omega u^\star\right|^2,
\]
since we have assumed that $\mathbb P_0(\Gamma)\subset X_h$. For
$D(A)$ we define the norm
\[
\|(u_h,u^\star)\|_{D(A)}^2 := \|(u_h,u^\star)\|_V^2+\|\Delta_\pm
u^\star\|_{\mathbb B\setminus\Gamma}^2.
\]
To define the operator $A:D(A)\to H$ associated to the evolution
problem \eqref{4.eq5}, we introduce the operators $\Delta_h^\kappa :
H^1(\Omega) \to V_h$ and $\gamma_h^t:H^{-1/2}(\Gamma)\to V_h$,
defined by the discrete equations
\begin{equation}\label{4.eq7}
(c^{-2}\Delta_h^\kappa u,v_h)_\Omega=-(\kappa\nabla u,\nabla
v_h)_\Omega \qquad \forall v_h \in V_h
\end{equation}
and
\begin{equation}\label{4.eq8}
(c^{-2}\gamma_h^t \lambda,v_h)_\Omega =\langle\lambda,\gamma
v_h\rangle_\Gamma \qquad \forall v_h \in V_h,
\end{equation}
respectively. These operators are defined so that \eqref{4.eq5a} can
be rewritten as $\ddot u_h= \Delta_h^\kappa u_h
-\gamma_h^t\jump{\partial_\nu u_h^{\mathrm{ext}}}$. The operator
$A:D(A)\to H$ is then defined by
\begin{equation}\label{4.eq21}
A(u_h,u^\star):= (\Delta_h^\kappa u_h-\gamma_h^t\jump{\partial_\nu
u^\star}, \Delta_\pm u^\star).
\end{equation}
Problem \eqref{4.eq5} has the general form \eqref{A.eq4}
 with initial data $(u_{h,0},0) \in D(A)$ and
$(v_{h,0},0)\in V$.

\paragraph{Rigid motions of the system.} Let $M:=\{ 0\}\times \mathrm{span}\{
\chi_\Omega\}$. Noting that
\[
\jump{\partial_\nu\chi_\Omega}=0,
\qquad\jump{\gamma\chi_\Omega}=1\in Y_h, \qquad
\gamma_{\partial\mathbb B}\chi_\Omega=0, \qquad
\partial_\nu^-\chi_\Omega=0, \quad\mbox{and}\quad
\gamma^+\chi_\Omega=0,
\]
it is simple to check that $M\subset D(A)$ and $M\subset \ker(A)$.
We consider the following seminorm in $V$
\[
|(u_h,u^\star)|_V^2:=\| \kappa^{1/2} \nabla u_h\|_\Omega^2+\|\nabla
u^\star\|_{\mathbb B\setminus\Gamma}^2,
\]
associated to a semi-inner product
$[(u_h,u^\star),(v_h,v^\star)]_V$. The hypotheses to consider the
finite dimensional space $M$ as a space of rigid motions of the
evolution problem \eqref{A.eq4} (see Section \ref{sec:A}) are then
easily verified.

\paragraph{Verification of the associated Green's Identity.} Let
$\underline u:=(u_h,u^\star)\in D(A)$ and $\underline
v_:=(v_h,v^\star)\in V$. Using the definition of the discrete
operators \eqref{4.eq7}-\eqref{4.eq8} and the definition of the weak
normal derivatives, it follows that
\begin{alignat*}{4}
(A\underline u,\underline v)_H+[\underline u,\underline v]_V & =
(c^{-2} \Delta_h^\kappa
u_h,v_h)_\Omega-(c^{-2}\gamma_h^t\jump{\partial_\nu
u^\star},v_h)_\Omega+(\Delta_\pm u^\star,v^\star)_{\mathbb B\setminus\Gamma}\\
& \quad +(\kappa\nabla u_h,\nabla v_h)_\Omega +(\nabla
u^\star,\nabla v^\star)_{\mathbb B\setminus\Gamma}\\
&=-\langle\jump{\partial_\nu u^\star},\gamma
v_h\rangle_\Gamma+\langle\partial_\nu^-
u^\star,\gamma^-v^\star\rangle_\Gamma-\langle\partial_\nu^+u^\star,\gamma^+v^\star\rangle_\Gamma\\
&=-\langle\underbrace{\jump{\partial_\nu u^\star}}_{\in
X_h},\underbrace{\gamma v_h-\gamma^+v^\star}_{\in
X_h^\circ}\rangle_\Gamma-\langle \underbrace{\partial_\nu^-
u^\star}_{\in Y_h^\circ},\underbrace{\jump{\gamma v^\star}}_{\in
Y_h}\rangle_\Gamma=0.
\end{alignat*}

\paragraph{The Yukawa double layer potential.} For the proof of the
surjectivity we need to introduce the double layer potential for the
Yukawa operator on $\Xi=\Gamma \cup \partial\mathbb B$
\[
\mathrm D\varphi:= \int_\Xi \nabla_{\mathbf y}E(\punto,\mathbf y)
\varphi(\mathbf y)\mathrm d\Xi(\mathbf y),
\]
where $E$ is given in \eqref{2.eq23}. This potential defines a
bounded operator
\begin{equation}\label{4.eq11}
\mathrm D:H^{1/2}(\Xi) \to H^1(\mathbb R^d\setminus\Xi)
\end{equation}
 such that
$\Delta (\mathrm D\varphi)=\mathrm D\varphi$ in $\mathbb
R^d\setminus\Xi$ for all $\varphi$. Two bounded integral operators
are associated to this potential
\[
\mathrm K:=\smallfrac12(\gamma^+_\Xi+\gamma^-_\Xi)\mathrm
D:H^{1/2}(\Xi) \to H^{1/2}(\Xi), \qquad \mathrm
W:=-\partial_{\nu,\Xi}^\pm\mathrm D:H^{1/2}(\Xi)\to H^{-1/2}(\Xi).
\]
The operator $\mathrm W$ is coercive \cite{McL00}
\begin{equation}\label{4.eq9}
\langle\mathrm W\varphi,\varphi\rangle_\Xi\ge C
\|\varphi\|_{1/2,\Xi}^2 \qquad \forall \varphi \in H^{1/2}(\Xi).
\end{equation}
Finally the adjoint of $\mathrm K$ satisfies $\mathrm
K^t=\smallfrac12 (\partial_{\nu,\Xi}^++\partial_{\nu,\Xi}^-)\mathrm
S$, where $\mathrm S$ is the single layer potential defined in
\eqref{2.eq25}.

\paragraph{Verification of the surjectivity property.} Let
$(f_h,f)\in V_h \times L^2(\mathbb B)=H$. As we did in Section
\ref{sec:2}, we start by finding $u^{\mathrm{nh}}\in H^2(\mathbb
B)\cap H^1_0(\mathbb B)$ such that $-\Delta
u^{\mathrm{nh}}+u^{\mathrm{nh}}=f$ in $\mathbb B$. We then look for
\begin{eqnarray*}
u_h \in V_h, & &  \lambda=(\lambda_h,\lambda_\partial)\in \underline
X_h:= X_h \times H^{-1/2}(\partial\mathbb B)\subset
H^{-1/2}(\Xi) \\
& & \varphi=(\varphi_h,0)\in \underline Y_h:=Y_h\times \{0\}\subset
H^{1/2}(\Xi)
\end{eqnarray*}
satisfying the equations:
\begin{subequations}\label{4.eq10}
\begin{alignat}{4}
\label{4.eq10a} (c^{-2} u_h,v_h)_\Omega+(\kappa\nabla u_h,\nabla
v_h)_\Omega -\langle \lambda_h,\gamma v_h\rangle_\Gamma
&=(c^{-2}f_h,v_h)_\Omega
&\qquad& \forall v_h \in V_h,\\
\label{4.eq10b} \langle\mu_h,\gamma u_h\rangle_\Gamma
+\langle\mu,\mathrm V\lambda\rangle_\Xi
-\langle\mu,\smallfrac12\varphi+\mathrm K\varphi\rangle_\Xi
&=\langle\mu_h,\gamma
u^{\mathrm{nh}}\rangle_\Gamma & & \forall \mu =(\mu_h,\mu_\partial)\in \underline X_h,\\
\label{4.eq10c} \langle \smallfrac12\lambda+\mathrm
K^t\lambda,\xi\rangle_\Xi+\langle\mathrm
W\varphi,\xi\rangle_\Xi&=\langle\partial_\nu
u^{\mathrm{nh}},\xi_h\rangle_\Gamma & & \forall \xi=(\xi_h,0) \in
\underline Y_h.
\end{alignat}
\end{subequations}
Note that because of the particular form of the space $\underline
Y_h$, the three bilinear forms where either $\xi$ or $\varphi$
appear are actually duality products in $\Gamma$. The bilinear form
of problem \eqref{4.eq10} is coercive in $H^1(\Omega)\times
H^{-1/2}(\Xi)\times H^{1/2}(\Xi)$ by \eqref{2.eq5} and
\eqref{4.eq9}. Therefore problem \eqref{4.eq10} has a unique
solution. The final step is the verification that the pair
\[
(u_h,u^\star):=(u_h,u^{\mathrm{nh}}+\mathrm D\varphi-\mathrm
S\lambda)
\]
belongs to $D(A)$ and that $(I-A)(u_h,u^\star)=(f_h,f)$. This
follows from several simple arguments that we next list. Because of
the potential form for $u^\star$ and the smoothness of
$u^{\mathrm{nh}}$ across $\Gamma$ it follows that
\begin{equation}\label{4.eq12}
\jump{\partial_\nu u^\star}=-\lambda_h \in X_h, \qquad \jump{\gamma
u^\star}=-\varphi_h \in Y_h.
\end{equation}
Also
\begin{equation}\label{4.eq14}
\Delta_\pm u^\star-u^\star=f
\end{equation}
and therefore $\Delta_\pm u^\star\in L^2(\mathbb B)$, while
$u^\star\in H^1(\mathbb B\setminus\Gamma)$ because
$u^{\mathrm{nh}}\in H^2(\mathbb B)$ and the mapping properties of
potentials \eqref{2.eq24}, \eqref{4.eq11} hold. Substituting
\eqref{4.eq12} in \eqref{4.eq10a} and using the definitions of the
discrete operators \eqref{4.eq7}-\eqref{4.eq8} it follows that
\begin{equation}\label{4.eq13}
u_h-\Delta_h^\kappa u_h+\gamma_h^t\jump{\partial_\nu u^\star}=f_h.
\end{equation}
If we test \eqref{4.eq10b} with elements $(0,\mu_\partial)\in
\{0\}\times H^{1/2}(\partial\mathbb B)$ it follows that
$\gamma_{\partial\mathbb B}u^\star=0$. If we test with $(\mu_h,0)\in
X_h \times\{0\}$, it follows that
\[
\gamma u_h-\gamma^+ u^\star=\gamma u_h-\mathrm
V\lambda-(\smallfrac12\varphi_h+\mathrm K\varphi)+\gamma
u^{\mathrm{nh}}\in X_h^\circ.
\]
Finally, equation \eqref{4.eq10c} is equivalent to asserting that
\[
\partial_\nu^+ u^\star=\partial_\nu
u^{\mathrm{nh}}-\smallfrac12\lambda_h+\mathrm K^t \lambda-\mathrm
W\varphi \in Y_h^\circ.
\]
The preceding arguments have shown that $(u_h,u^\star)\in D(A)$
while \eqref{4.eq14}-\eqref{4.eq13} proves that
$(I-A)(u_h,u^\star)=(f_h,f)$.

\paragraph{Conclusion.} We can now apply the theory for Cauchy
problems exposed in Section \ref{sec:A}. The statements of Theorem
\ref{3.th1} are a direct consequence of these results.

\section{Two simple extensions}\label{sec:5}

We now show two other situations where the techniques developed in
the preceding sections can be applied with minor modifications.

\subsection{Double layer potentials for sound-hard scattering}

Let us consider again the geometrical setting of Section
\ref{sec:1}. The problem of sound-hard scattering by an obstacle
occupying the region $\overline\Omega$ can be expressed with the
equations \eqref{1.eq1}-\eqref{1.eq1b} with the boundary condition
\eqref{1.eq1.b} substituted by $\partial_\nu u=0$ on $\Gamma$ for
all $t$. The assumptions on the initial data are the same as those
given in Section \ref{sec:1}. If we define $u^{\mathrm{free}}$ with
\eqref{1.eq3}-\eqref{1.eq3a}, the solution to this problem can be
expressed as $u=\mathcal D*\varphi+u^{\mathrm{free}}$, where
$\varphi:\mathbb R \to H^{1/2}(\Gamma)$ vanishes identically for
negative values of $t$.

The semidiscrete formulation follows from choosing a discrete space
\[
\mathbb P_0(\Gamma) \subset Y_h \subset H^{1/2}(\Gamma),
\]
discretizing the boundary condition
\begin{equation}\label{5.eq1}
\left[\begin{array}{l} \varphi_h:\mathbb R \to Y_h, \qquad
\varphi_h\equiv 0 \mbox{ in $(-\infty,0)$},\\[1.5ex]
\langle -\mathcal W*\varphi_h+\partial_\nu
u^{\mathrm{free}},\xi_h\rangle_\Gamma =0 \qquad \forall \xi_h \in
Y_h \quad \forall t,
\end{array}
\right.
\end{equation}
and proposing
\begin{equation}\label{5.eq2}
u_h = \mathcal D*\varphi_h + u^{\mathrm{free}}
\end{equation}
as approximation of $u$. Note that the following transmission
conditions are satisfied for all $t\ge 0$:
\[
\jump{\gamma u_h}\in Y_h, \qquad \jump{\partial_\nu u_h}=0, \qquad
\partial_\nu u_h \in Y_h^\circ.
\]

\begin{theorem}\label{5.th1}
The semidiscrete total field $u_h$ and the associated density
$\varphi_h$ given by \eqref{5.eq1} and \eqref{5.eq2} satisfy
\begin{eqnarray}\label{5.eq3}
u_h &\in& \mathcal C^2([0,\infty); L^2(\mathbb R^d))\cap \mathcal
C^1([0,\infty);H^1(\mathbb R^d\setminus\Gamma)),
\\
\label{5.eq4} \varphi_h &\in& \mathcal C^1([0,\infty);
H^{1/2}(\Gamma)).
\end{eqnarray}
Moreover, the energy
\begin{equation}\label{5.eq5}
\smallfrac12 \| \nabla u_h(t)\|_{\mathbb R^d\setminus\Gamma}^2
+\smallfrac12 \|\dot u_h(t)\|_{\mathbb R^d}^2
\end{equation}
is constant over time.
\end{theorem}

\begin{proof} The techniques for the proof of this result are very similar to those of Section \ref{sec:3}.
We will simply sketch the main steps. First of all we pick a fixed
time interval $[0,T]$ and choose a ball $\mathbb B$ that contains
the support of the solution for all $t \in [0,T]$. It is now more
convenient to include a Neumann boundary condition on the cut-off
boundary $\partial_{\nu,\partial\mathbb B}u_h=0$ for all $t$. The
spaces for the formulation as a Cauchy problem are the following:
\begin{eqnarray}\label{5.eq21}
H &:=& L^2(\mathbb B),\\
\label{5.eq22} V &:=& \{ u \in H^1(\mathbb
B\setminus\Gamma)\,:\,\jump{\gamma u}\in
Y_h\},\\
\label{5.eq23}D(A) &:=& \{ u \in V\,:\, \Delta_\pm u \in L^2(\mathbb
B),\quad \jump{\partial_\nu u}=0, \quad \partial_\nu u\in
Y_h^\circ,\quad
\partial_{\nu,\partial\mathbb B}u=0\}.
\end{eqnarray}
The relevant norms and seminorms are:
\[
\| u\|_H:=\| u\|_{\mathbb B}, \quad |u|_V:=\|\nabla u\|_{\mathbb
B\setminus\Gamma}, \quad \| u\|_V^2:=
|u|_V^2+\left|\frac1{|\Omega|}\int_\Omega u\right|^2+
\left|\frac1{|\mathbb B\cap\Omega^+|}\int_{\mathbb B
\cap\Omega^+}u\right|^2.
\]
The associated operator is the same as in Section \ref{sec:3},
namely $A:=\Delta_\pm$. The space of associated rigid motions is
two-dimensional $M:=\mathrm{span}\{ \chi_\Omega, \chi_{\mathbb B\cap
\Omega^+}\}$. The proof of the corresponding Green's Identity is
straightforward. For surjectivity we proceed in two steps. Given
$f\in L^2(\mathbb B)$, we first choose $u^{\mathrm{nh}}\in
H^2(\mathbb B)$ such that $ -\Delta
u^{\mathrm{nh}}+u^{\mathrm{nh}}=f$ in $\mathbb B$ and
$\partial_{\nu,\partial\mathbb B}u^{\mathrm{nh}}=0$. (Note that this
is possible by basic regularity theorems of elliptic problems on
smooth domains.) Next we look for
$\varphi=(\varphi_h,\varphi_\partial)\in \underline Y_h:=Y_h \times
H^{1/2}(\partial\mathbb B)\subset H^{1/2}(\Xi)$ that solves the
coercive variational problem
\[
\left[\begin{array}{l} \varphi \in \underline Y_h, \\[1.5ex]
\langle \mathrm W \varphi,\xi\rangle_\Gamma=\langle \partial_\nu
u^{\mathrm{nh}},\xi_h\rangle_\Gamma, \qquad \forall
\xi=(\xi_h,\xi_\partial) \in \underline Y_h,
\end{array}\right.
\]
and define $u=u^{\mathrm{nh}}+\mathrm D \varphi$. (Notations for
Yukawa potentials and operators are those of Section \ref{sec:4}.)
It is simple to prove that $u \in D(A)$ and $u-Au=f$, which finishes
the proof.
\end{proof}

\subsection{A tighter transparent boundary condition}\label{sec:5.2}

We now consider the transmission problem
\eqref{3.eq1}-\eqref{3.eq1comp} of Section \ref{sec:3}. Instead of
the transparent boundary condition that uses approximations of both
Cauchy data on $\Gamma$ it is possible to construct another
boundary-field formulation in the spirit of the BEM-FEM coupling
schemes of Martin Costabel \cite{Cos87} and Houde Han \cite{Han90}.
The boundary unknown is $\lambda:=\partial_\nu u^{\mathrm{ext}}$.
The identities of Cauchy data on the boundary and the transmission
conditions \eqref{3.eq1c}-\eqref{3.eq1d} can be used to produce the
following equations:
\begin{equation}\label{5.eq6}
\smallfrac12\gamma u-\mathcal K*\gamma u+\mathcal V*\lambda=0 \qquad
-\partial_\nu^\kappa u=\mathcal W*\gamma
u-\smallfrac12\lambda+\mathcal K^t*\lambda.
\end{equation}
In \eqref{5.eq6} and all similar future expressions, it will be
understood that $\gamma u(t)\equiv 0$ for $t<0$, although $u(t)$
will be defined for $t\ge 0$ and we will not look for a smooth
continuation of $u$ for negative values of $t$.

The boundary-field formulation looks for $u:[0,\infty) \to
H^1(\Omega)$ and $\lambda:\mathbb R \to H^{-1/2}(\Gamma)$ satisfying
initial conditions
\[
u(0)=u_0, \qquad \dot u(0)=v_0, \qquad \lambda\equiv 0 \mbox{ in
$(-\infty,0)$},
\]
and the equations
\begin{subequations}\label{5.eq7}
\begin{alignat}{4}\nonumber
 (c^{-2} \ddot u,v)_\Omega+(\kappa\nabla u,\nabla
v)_\Omega+\langle\mathcal W*\gamma u,\gamma
v\rangle_\Gamma\hspace{2cm}& &
&\\-\langle\smallfrac12\lambda-\mathcal K^t*\lambda,\gamma
v\rangle_\Gamma =0 & & \quad&
\forall v \in H^1(\Omega),\label{5.eq7a}\\
\label{5.eq7b} \langle\mu,\smallfrac12\gamma u-\mathcal K*\gamma
u\rangle_\Gamma + \langle\mu,\mathcal V*\lambda\rangle_\Gamma=0 & &
& \forall \mu \in H^{-1/2}(\Gamma),
\end{alignat}
\end{subequations}
for all $t$. The discrete version of these equations uses two spaces
\[
\mathbb P_0(\Omega)\subset V_h \subset H^1(\Omega), \qquad \mathbb
P_0(\Gamma)\subset X_h \subset H^{-1/2}(\Gamma),
\]
and looks for $u_h:[0,\infty) \to V_h$ and $\lambda_h: \mathbb R\to
X_h$ such that
\begin{equation}\label{5.eq8}
u_h(0)=u_{h,0}, \qquad \dot u_{h}(0)=v_{h,0}, \qquad \lambda_h\equiv
0 \mbox{ in $(-\infty,0)$},
\end{equation}
and for all $t$:
\begin{subequations}\label{5.eq9}
\begin{alignat}{4}\nonumber
(c^{-2} \ddot u_h,v_h)_\Omega+(\kappa\nabla u_h,\nabla
v_h)_\Omega+\langle\mathcal W*\gamma u_h,\gamma
v_h\rangle_\Gamma\hspace{2cm}& &
&\\-\langle\smallfrac12\lambda_h-\mathcal K^t*\lambda_h,\gamma
v_h\rangle_\Gamma =0 & & \quad&
\forall v_h \in V_h,
\label{5.eq9a} \\
\label{5.eq9b} \langle\mu_h,\smallfrac12\gamma u_h-\mathcal K*\gamma
u_h\rangle_\Gamma + \langle\mu_h,\mathcal
V*\lambda_h\rangle_\Gamma=0 & & & \forall \mu_h \in X_h.
\end{alignat}
\end{subequations}
Compared with the semidiscrete system \eqref{3.eq8}, this system has
one less group of equations and unknowns. The price to pay is the
fact that $u_h$ is affected by integral delay operators. The
reconstructed exterior solution is given by
\begin{equation}\label{5.eq10}
u_h^{\mathrm{ext}}:=\mathcal D*\gamma u_h-\mathcal S*\lambda_h.
\end{equation}
As  in Section \ref{sec:3}, we have to consider $u_h^{\mathrm{ext}}$
defined in $\mathbb R^d$ (and not only in $\Omega^+$) to have the
correct balance of energy.

\begin{theorem}\label{5.th2}
The semidiscrete total fields $(u_h,u_h^{\mathrm{ext}})$ and the
approximation to the normal derivative on the interface $\lambda_h$
given by \eqref{5.eq8}-\eqref{5.eq9}-\eqref{5.eq10} satisfy
\begin{eqnarray}
u_h & \in & \mathcal C^2([0,\infty); L^2(\Omega))\cap \mathcal
C^1([0,\infty);H^1(\Omega)),\\
u_h^{\mathrm{ext}} & \in & \mathcal C^2([0,\infty);L^2(\mathbb R^d)
\cap \mathcal C^1([0,\infty); H^1(\mathbb R^d\setminus\Gamma)),\\
\lambda_h &\in & \mathcal C([0,\infty);H^{-1/2}(\Gamma)).
\end{eqnarray} Moreover, the energy
\begin{equation}
\smallfrac12 \|\kappa^{1/2}\nabla u_h(t)\|_\Omega^2+\smallfrac12 \|
\nabla u_h^{\mathrm{ext}}(t)\|_{\mathbb R^d\setminus\Gamma}^2+
\smallfrac12\| c^{-1} \dot u_h(t)\|_\Omega^2+\smallfrac12\|\dot
u_h^{\mathrm{ext}}(t)\|_{\mathbb R^d\setminus\Gamma}
\end{equation}
is constant over time.
\end{theorem}

\begin{proof}
The proof has a similar structure as that of Theorem \ref{3.th1}. We
will just point out the milestones of the proof. Choosing a ball
$\mathbb B$ that contains the support of the solution for all $t\in
[0,T]$ and adding a Dirichlet boundary condition on $\partial\mathbb
B$, we can consider a Cauchy problem satisfied by the pair
$(u_h,u_h^{\mathrm{ext}})$ with the following elements: the spaces
\begin{eqnarray*}
H &:=& V_h\times L^2(\mathbb B),\\
V &:=& \{ (u_h,u^\star) \in H\,:\jump{\gamma u^\star}+\gamma u_h=0,
\quad \gamma^- u^\star\in X_h^\circ,\quad \gamma_{\partial\mathbb
B} u^\star=0\},\\
D(A)&:=& \{ (u_h,u^\star)\in V\,:\, \Delta_\pm u^\star \in
L^2(\mathbb B), \quad \jump{\partial_\nu u^\star}\in X_h\},
\end{eqnarray*}
the norms
\[
\|(u_h,u^\star)\|_H^2:=\|c^{-1}u_h\|_\Omega^2+\|u^\star\|_{\mathbb
B}^2 \qquad \|(u_h,u^\star)\|_V^2:=\|\kappa^{1/2}\nabla
u_h\|_\Omega^2+\|\nabla u^\star\|_{\mathbb B\setminus\Gamma}^2
\]
and the operator
\[
A(u_h,u^\star):=(\Delta_h^\kappa u_h+\gamma_h^t \partial_\nu^+
u^\star, \, \Delta_\pm u^\star)
\]
(see \eqref{4.eq7} and \eqref{4.eq8}). The associated Green's
Identity is straightforward to prove. To show surjectivity of $I-A$,
we proceed as follows. Given $(f_h,f) \in H$, we first take
$u^{\mathrm{nh}}\in H^2(\mathbb B)\cap H^1_0(\mathbb B)$ satisfying
$-\Delta u^{\mathrm{nh}}+u^{\mathrm{nh}}=f$ in $\mathbb B$. We then
construct
\begin{equation}\label{5.eq11}
(u_h,u^\star)=(u_h,u^{\mathrm{nh}}+\mathrm D \widetilde{\gamma
u_h}-\mathrm S\lambda)
\end{equation}
where
\[
u_h \in V_h, \qquad \lambda=(\lambda_h,\lambda_\partial) \in
\underline X_h:= X_h \times H^{-1/2}(\partial\mathbb B)\subset
H^{-1/2}(\Xi)
\]
is the solution of
\begin{subequations}\label{5.eq12}
\begin{alignat}{4}\nonumber
(c^{-2} u_h,v_h)_\Omega+(\kappa \nabla u_h,\nabla v_h)_\Omega
\hspace{1.5cm}& & &
\\
+\langle \mathrm W\widetilde{\gamma u_h},\widetilde{\gamma
v_h}\rangle_\Xi -\langle \smallfrac12\lambda-\mathrm
K^t\lambda,\widetilde{\gamma v_h}\rangle_\Xi & = (c^{-2}
f_h,v_h)_\Omega+\langle \partial_\nu u^{\mathrm{nh}},\gamma
v_h\rangle_\Gamma \qquad  \forall v_h \in V_h,\\
\langle \mu,\smallfrac12\widetilde{\gamma u_h}-\mathrm
K\widetilde{\gamma u_h}\rangle_\Xi +\langle \mu,\mathrm
V\lambda\rangle_\Xi &=\langle\mu_h,\gamma
u^{\mathrm{nh}}\rangle_\Gamma \qquad \forall
\mu=(\mu_h,\mu_\partial)\in \underline X_h,
\end{alignat}
\end{subequations}
and $\widetilde{\gamma u_h}=(\gamma u_h,0)\in H^{1/2}(\Gamma) \times
\{ 0\}\subset H^{1/2}(\Xi)$. Note that problem \eqref{5.eq12} is
associated to a coercive bilinear form and is therefore uniquely
solvable. It is then simple to prove that $(u_h,u^\star)\in D(A)$
and $A(u_h,u^\star)=(f_h,f)$, which finishes the proof.
\end{proof}

\section{Two negative results}\label{sec:6}

This last section shows two results where Galerkin discretizations
lead to problems with non-constant energy. This fact does not mean
that the discretizations are not valid (they can still be stable),
but at least shows how delicate the energy balance is when
discretized integral operators are used.

\subsection{A direct method for sound-soft
scattering}\label{sec:6.1}

Let us consider again the problem of Section \ref{sec:1}. The
solution of \eqref{1.eq1}-\eqref{1.eq1b} can be decomposed using the
free wave (solution to \eqref{1.eq3}-\eqref{1.eq3a}) and Kirchhoff's
formula for the scattered wave:
\begin{equation}\label{6.eq1}
u=u^{\mathrm{free}}-\mathcal D*\gamma u^{\mathrm{free}}-\mathcal
S*\lambda, \qquad \lambda:=\partial_\nu^+ u-\partial_\nu
u^{\mathrm{free}},
\end{equation}
where $\lambda:\mathbb R \to H^{-1/2}(\Gamma)$ is causal and in the
potential expression $\mathcal D*\gamma u^{\mathrm{free}}$ we have
to understand that $\gamma u^{\mathrm{free}}:\mathbb R \to
H^{1/2}(\Gamma)$ is a causal function even if $u^{\mathrm{free}}$ is
only defined for positive values of $t$. The indirect decomposition
\eqref{1.eq6} used $u=\mathcal S*\psi+u^{\mathrm{free}}$ was
naturally extended to $\mathbb R^d$, with the result that
$u(t)\equiv 0$ in $\Omega$ for all $t$. The extension to the
interior domain of \eqref{6.eq1} is
\[
u=u^{\mathrm{free}} \chi_{\Omega^+}-\mathcal D*\gamma
u^{\mathrm{free}}-\mathcal S*\lambda,
\]
because the potential expression $\mathcal D*\gamma
u^{\mathrm{free}}+\mathcal S*\lambda$ vanishes identically in
$\Omega$ by Kirchhoff's formula. The discrete version of this
process computes
\begin{equation}\label{6.eq2}
\left[\begin{array}{l}\lambda_h:\mathbb R \to X_h \qquad
\lambda_h\equiv 0 \mbox{ in $(-\infty,0)$},\\[1.5ex]
\langle \mu_h,\mathcal V*\lambda_h\rangle_\Gamma
=\langle\mu_h,\smallfrac12\gamma u^{\mathrm{free}}-\mathcal K*\gamma
u^{\mathrm{free}}\rangle_\Gamma \qquad \forall \mu_h \in X_h, \quad
\forall t,
\end{array}\right.
\end{equation}
and then constructs the total field
\begin{equation}\label{6.eq3}
u_h:=u^{\mathrm{free}}\chi_{\Omega^+}-\mathcal D*\gamma
u^{\mathrm{free}}-\mathcal S*\lambda_h.
\end{equation}
Note that $u_h$ satisfies the {\em non--homogeneous} transmission
problem
\begin{subequations}\label{6.eq4}
\begin{alignat}{4}
\ddot u_h=\Delta_\pm u_h, & & &\\
\jump{\gamma u_h}=0, \\
\gamma u_h \in X_h^\circ, \\
\jump{\partial_\nu u_h}+\partial_\nu u^{\mathrm{free}}\in X_h.
\end{alignat}
\end{subequations}
(Compare with \eqref{2.eq4} and note that the condition at
$\partial\mathbb B$ can always be added for finite time intervals.)
The fact that this problem is a non--homogeneous version of a
problem that is conservative gives a first hint that the natural
energy of this problem will not be constant. Also, applying
integration by parts and \eqref{6.eq4}, we can prove that
\[
\frac{\mathrm d}{\mathrm dt}\Big(\frac12\| \dot u_h\|_{\mathbb
R^d}^2+\frac12\|\nabla u_h\|_{\mathbb
R^d\setminus\Gamma}^2\Big)=-\langle \partial_\nu
u^{\mathrm{free}},\gamma \dot u_h\rangle_\Gamma.
\]
This shows that energy is not constant.

\subsection{Transparent conditions with one equation}\label{sec:6.2}

Consider again the propagation problem in free space \eqref{3.eq1}.
Instead of using two integral identities as in Sections \ref{sec:3}
we can work with a single integral equation in the spirit of the
one-equation coupling of BEM-FEM of Claus Johnson and Jean-Claude
N\'{e}d\'{e}lec \cite{JohNed80, Say09}. After space Galerkin discretization,
the coupled system becomes an evolution problem that looks for
$u_h:[0,\infty) \to V_h$ and $\lambda_h:\mathbb R\to X_h$ such that
\[
u_h(0)=u_{h,0}, \qquad \dot u_h(0)=v_{h,0}, \qquad \lambda_h \equiv
0 \mbox{ in $(-\infty,0)$},
\]
and for all $t\ge 0$
\begin{subequations}
\begin{alignat}{4}
(c^{-2}\ddot u_h,v_h)_\Omega+(\kappa\nabla u_h,\nabla
v_h)-\langle\lambda_h,\gamma v_h\rangle_\Gamma =0 & \qquad & &
\forall v_h\in V_h,\\
\langle \mu_h,\smallfrac12 \gamma u_h-\mathcal K*\gamma
u_h\rangle_\Gamma +\langle\mu_h,\mathcal V*\lambda_h\rangle_\Gamma
=0 & & & \forall \mu_h \in X_h.
\end{alignat}
\end{subequations}
Note that the first of these equations coincides with \eqref{3.eq8a}
(the first discrete equation of the method in Section \ref{sec:3}),
while the second one is \eqref{5.eq9b} (the second equation in the
method of Section \ref{sec:5.2}). The discrete exterior solution is
defined with \eqref{5.eq10}. The same kind of manipulations that we
have been applied above shows that the pair
$(u_h,u_h^{\mathrm{ext}})$ satisfies
\begin{alignat*}{4}
\ddot u_h =\Delta_h^\kappa u_h-\gamma_h^t \jump{\partial_\nu
u_h^{\mathrm{ext}}},\\
\ddot u_h^{\mathrm{ext}}=\Delta_\pm u_h^{\mathrm{ext}},\\
\jump{\gamma u_h^\mathrm{ext}}+\gamma u_h=0,\\
\gamma^- u_h^{\mathrm{ext}}\in X_h^\circ,\\
\jump{\partial_\nu u_h^{\mathrm{ext}}}\in X_h,
\end{alignat*}
with the discrete operators defined in \eqref{4.eq7}-\eqref{4.eq8}.
It is simple to see that the transmission conditions above lead to a
choice of spaces like \eqref{5.eq21}-\eqref{5.eq23}, while the
operator itself is \eqref{4.eq21}. This mismatch between domain of
the operator and operator leads to lack of energy conservation,
namely, for smooth solutions
\[
\frac{\mathrm d}{\mathrm dt}\Big(\frac12\| c^{-1}\dot u_h
\|_\Omega^2+\frac12\|\kappa^{1/2}\nabla u_h\|_\Omega^2+\frac12\|
\dot u_h^{\mathrm{ext}}\|_{\mathbb
R^d\setminus\Gamma}^2+\frac12\|\nabla u_h^{\mathrm{ext}}\|_{\mathbb
R^d\setminus\Gamma}^2\Big)=\langle \partial_\nu^-
u_h^{\mathrm{ext}},\jump{\gamma \dot
u_h^{\mathrm{ext}}}\rangle_\Gamma.
\]
Furthermore, by recasting the evolution problem as a first order
system, it is possible to show that the corresponding operator
$\mathcal A$ in \eqref{A.eq21} is not maximal dissipative (cf.
\cite[Chapter 4]{Kes89}) and therefore it cannot be the
infinitesimal generator of a contractive strongly continuous
semigroup. As in the case of the discrete Kirchhoff formula of
Section \ref{sec:6.1}, this does not mean that the discretization
leads to an unstable method, but it is at least a hint that some
problems might arise due to the lack of energy conservation.

\section{Some abstract arguments about wave equations}\label{sec:A}

In this section we summarize the abstract results on Cauchy problems
for second order equations that we have used throughout the article.
The results are elementary consequences of the theory of strongly
continuous groups of isometries in the Hilbert setting that can be
found in basic texts as \cite{EngNag06} or \cite{Kes89}.

\paragraph{Cauchy problems for abstract wave equations.}
Let us consider three Hilbert spaces with continuous inclusions
\begin{equation}\label{A.eq1}
D(A)\subset V \subset H.
\end{equation}
The inner produces and norms of $V$ and $H$ will be recognized with
the name of the space as a subscript. Let $A:D(A) \to H$ be a
bounded linear operator satisfying:
\begin{itemize}
\item[(a)] an abstract Green's identity
\begin{equation}\label{A.eq2}
(A u,v)_H+(u,v)_V=0 \qquad \forall u\in D(A)\quad v \in V
\end{equation}
\item[(b)] a surjectivity condition
\begin{equation}\label{A.eq3}
I-A: D(A) \to H \mbox{ is onto,}
\end{equation}
$I$ being the inclusion operator.
\end{itemize}
Then, for arbitrary $(u_0,v_0)\in D(A)\times V$, the initial value
problem
\begin{subequations}\label{A.eq4}
\begin{alignat}{4}
\ddot u=A u & & \qquad & t \ge 0,\\
u(0)=u_0, \\
\dot u(0)=v_0,
\end{alignat}
\end{subequations}
has a unique solution
\begin{equation}\label{A.eq5}
u \in \mathcal C^2([0,\infty);H) \cap \mathcal C^1([0,\infty);V)\cap
\mathcal C([0,\infty);D(A)).
\end{equation}
Moreover, the energy
\[
e(t):= \smallfrac12 \| \dot u(t)\|_H^2 +\smallfrac12 \| u(t)\|_V^2
\]
is constant as a function of $t \in [0,\infty)$. Also, the
injections \eqref{A.eq1} are dense. The above result can be proved
by considering the unbounded operator in $V\times H$
\begin{equation}\label{A.eq21}
\mathcal A:=\left[\begin{array}{cc} 0 & I \\ A & 0
\end{array}\right] : D(\mathcal A)   \to V \times
H, \qquad D(\mathcal A) :=D(A)\times V \subset V\times H
\end{equation}
and showing that $\pm\mathcal A$ are maximal dissipative and
therefore $\mathcal A$ is the infinitesimal generator of a
$C_0-$group of isometries in the Hilbert space $V\times H$. The
possibility of reducing all the hypotheses to properties that have
to be satisfied by $A$ is related to the fact that $D(\mathcal A)$
is the product space $D(A)\times V$, where the second space is the
same as the first space in the Hilbert space $V\times H$ where the
problem is set.

\paragraph{Problems displaying rigid motions.} Assume that we have
three spaces in the same conditions above and that in $V$ we also
have a seminorm $|\punto|_V$, proceeding from a semi-inner product
$[\punto,\punto]_V$. A space of rigid motions of the system
\eqref{A.eq4} is a {\em finite dimensional space} $M$ such that
\[
M \subset D(A), \qquad M \subset \ker (A), \qquad |m|_V=0 \quad
\forall m\in M,
\]
and, if $P:H\to M$ denotes the orthogonal projection onto $M$,
\begin{equation}\label{A.eq8}
C_1\| v\|_V^2 \le |v|_V^2+\|Pv\|_H^2 \le C_2\|v\|_V^2 \qquad \forall
v \in V.
\end{equation}
In particular, this implies that $|v|_V=0$ if and only if $v\in M$.
The energy of the system is now measured in the following form
\begin{equation}\label{A.eq6}
e(t):= \smallfrac12\|\dot u(t)\|_H^2+\smallfrac12 |u(t)|_V^2.
\end{equation}
If we start with initial conditions $(m,0)\in M\times \{0\}$, the
solution remains constant over time and energy vanishes. If we start
with $(0,m)\in \{0\}\times M$, then the solution is $u(t)=m\, t$.
This solution has only kinetic energy and potential energy vanishes
identically
\[
e(t)=\smallfrac12 \| m\|_H^2.
\]
The conditions are:
\begin{itemize}
\item[(a')] a (modified) abstract Green identity
\begin{equation}\label{A.eq7}
(A u,v)_H +[ u,v]_V=0 \qquad \forall u\in D(A) \quad v \in V,
\end{equation}
\item[(b)] the surjectivity condition $I-A:D(A)\to H$ is onto.
\end{itemize}
With these hypotheses, problem \eqref{A.eq4} has a unique solution
with the regularity of \eqref{A.eq5} for any $(u_0,v_0)\in
D(A)\times V$. Energy, defined with \eqref{A.eq6}, is constant in
time.

\begin{remark}\rm To see how the frame of evolution equations with
rigid motions fits into the general frame, we need to consider the
spaces
\[
H_0:=\{ u\in H\,:\, Pu=0\}=\{ u\in H\,:\,(u,m)_H=0\quad \forall m
\in M\},
\]
\[
V_0:=V\cap H_0, \qquad D(A_0):=D(A)\cap H_0.
\]
In $V_0$ we consider the norm $|\punto|_V$ (see \eqref{A.eq8}). The
operator $A$ has range in $H_0$, since \eqref{A.eq7} implies that
\[
0=(Au,m)_H+[u,m]_V=(Au,m)_H \qquad \forall m \in M.
\]
Finally, the surjectivity condition implies that $I-A_0:D(A_0)\to H$
is surjective by the same argument.
\end{remark}

\paragraph{Waves in free space.} Problem \eqref{1.eq3}-\eqref{1.eq3a}, that deals
with propagation of initial conditions by the wave equation in free
space does not fit in the simple frame of this section unless we
apply some kind of cut-off argument. The main difficulty stems from
the definition of the energy space.
\begin{itemize}
\item[(1)] It is well known that the closure of the space of smooth
compactly supported functions with the norm $\|\nabla
\punto\|_{\mathbb R^3}$ is not a subset of $L^2(\mathbb R^3)$.
Instead, the resulting space can be characterized as a weighted
Sobolev space (see \cite{Han71, GalRio85, AmrGirGir94}). The simple
frame with three spaces ($D(A)$ as the domain of the operator, $V$
as the potential energy space and $H$ as the space where kinetic
energy is measured) cannot be used for the corresponding second
order Cauchy problem. Instead, the problem has to be rewritten as a
first order system with four Hilbert spaces involved (compare with
\eqref{A.eq21}). Conditions on the operator need to be written in a
much more complicated form in order to show existence, uniqueness
and energy conservation.
\item[(2)] The two dimensional case is even more involved. As
explained in \cite{GalRio85}, the closure of the space of smooth
compactly supported functions with $\|\nabla\punto\|_{\mathbb R^2}$
cannot be understood as a space of functions in any natural way.
This adds another complication to the four space setting that is
needed in the three dimensional case.
\end{itemize}
Since the present work deals with compactly supported initial data
and speed of propagation of waves is finite, the strategy of cutting
off the analytical domain can be applied for any finite time
interval, which is enough for our purposes.

\bibliographystyle{abbrv}
\bibliography{RefsTDIE}

\end{document}